\documentclass[12pt]{amsart}
\usepackage{amssymb,amsmath,amsthm,latexsym,mathtools,dsfont}
\usepackage{mathrsfs}
\usepackage{mdwlist}
\usepackage{graphicx}
\usepackage{xcolor}
\usepackage{enumerate}
\usepackage[shortlabels]{enumitem}
\usepackage{a4wide}
\input xy
\xyoption{all}
\usepackage{physics}
\definecolor{amber}{rgb}{1.0, 0.75, 0.0}
\definecolor{my_green}{RGB}{0,142,0}
\definecolor{my_blue}{RGB}{0,51,153}
\definecolor{color1}{RGB}{0,51,102}
\definecolor{oranje}{RGB}{255, 153, 51}
\newcommand{\R}{\mathbb{R}}

\newcommand{\N}{\mathbb{N}}

\newcommand{\Bb}{\mathcal{B}}

\newcommand{\e}{\varepsilon}
\newcommand{\w}{\widetilde}

\newcommand{\Ma}{\mathcal{M}}
\newcommand{\Ta}{\mathcal{T}}

\newcommand{\n}[1]{\|#1\|}
\newcommand{\nn}[1]{{\vert\kern-0.25ex\vert\kern-0.25ex\vert #1 
    \vert\kern-0.25ex\vert\kern-0.25ex\vert}}
\newcommand{\lnn}[1]{{\left\vert\kern-0.25ex\left\vert\kern-0.25ex\left\vert #1 
    \right\vert\kern-0.25ex\right\vert\kern-0.25ex\right\vert}}

\newcommand{\supp}{\mathrm{supp}}
\newcommand{\ccup}{\scalebox{0.85}{$\bigcup$}}
\newcommand{\dast}{{\ast\ast}}

\renewcommand{\leq}{\leqslant}
\renewcommand{\geq}{\geqslant}

\definecolor{darkgreen}{RGB}{0, 153, 51}
\definecolor{violet}{RGB}{112, 73, 170}
\definecolor{darkred}{RGB}{153, 0, 0}
\definecolor{darkdarkblue}{RGB}{0, 0, 102}
\definecolor{darkblue}{RGB}{153, 204, 255}
\definecolor{bluee}{RGB}{204, 230, 255}
\definecolor{bluee2}{RGB}{128, 191, 255}
\definecolor{shadow}{RGB}{82, 122, 122}
\definecolor{my_green1}{RGB}{0, 153, 0}
\definecolor{my_orange}{RGB}{255, 204, 0}
\definecolor{my_yellow}{RGB}{255, 255, 102}
\newtheorem{theorem}{Theorem}
\newtheorem{lemma}[theorem]{Lemma}
\newtheorem{proposition}[theorem]{Proposition}

\newtheorem*{theoremA}{Theorem}
\theoremstyle{definition}

\theoremstyle{remark}

\title[When is the Szlenk derivation of a~dual unit ball another ball?]{When is the Szlenk derivation of a dual unit ball another ball?}
\author{Tomasz Kochanek and Marek Miarka}
\address{Institute of Mathematics, University of Warsaw, Banacha~2, 02-097 Warsaw, Poland}
\email{tkoch@mimuw.edu.pl, m.miarka@uw.edu.pl}
\thanks{The work has been supported by the National Science Centre grant no. 2020/37/B/ST1/01052.}
\begin{document}
\begin{abstract}
We show that if a separable Banach space has Kalton's property $(M^\ast)$, then all $\e$-Szlenk derivations of the dual unit ball are balls, however, in the case of the dual of Baernstein's space, all those Szlenk derivations are balls having the same radius as for $\ell_2$, yet this space fails property $(M^\ast)$. By estimating the radii of enveloping balls, we show that the Szlenk derivations are not balls for Tsirelson's space and the dual of Schlumprecht's space. Using the Karush--Kuhn--Tucker theorem we prove that the same is true for the duals of certain sequential Orlicz spaces. 
\end{abstract}
\maketitle

\section{Introduction}
Let $X$ be an infinite-dimensional Banach space. For any weak$^\ast$-compact set $K\subset X^\ast$ and any $\e\in (0,2)$, the $\e$-{\it Szlenk derivation} of $K$ is defined by
\begin{equation}\label{szlenk_def}
s_\e K=\bigl\{x^\ast\in K\colon \mathrm{diam} (K\cap V)>\e\mbox{ for every }w^\ast\mbox{-open neighborhood }V\mbox{ of }x^\ast\bigr\}.
\end{equation}
Iterating this procedure for $K$ being the dual unit ball $B_{X^\ast}=\{x^\ast\in X^\ast\colon\n{x^\ast}\leq 1\}$ gives rise to the $\e$-Szlenk index $Sz(X,\e)=\min\{\alpha\colon s_\e^\alpha B_{X^\ast}=\varnothing\}$, and the Szlenk index $Sz(X)=\sup_{\e>0}Sz(X,\e)$ introduced in \cite{szlenk}, which is one of the most classical ordinal indices in Banach space theory. The fundamental fact, going back to Namioka and Phelps \cite{NP}, says that $Sz(X)$ is a~well-defined ordinal number if and only if $X$ is an~Asplund space (see also \cite[\S 1.5]{DGZ}). Although the behavior of Szlenk indices has been widely investigated from various perspectives (see, e.g., \cite{causey}, \cite{JFA}, \cite{PAMS}, \cite{GKL}, \cite{KOS}, and the survey \cite{lancien}), the precise form of $s_\e B_{X^\ast}$ is in general difficult to describe. In the present paper, our goal is to provide (at least a~partial) answer to the question when does $s_\e B_{X^\ast}$ happen to be a~ball for some, or all, $\e\in (0,2)$. This is inherently connected to investigating the radii of the optimal enveloping balls of the Szlenk derivations of the dual unit ball which are defined by the formulas:
$$
r_X(\e)=\sup\bigl\{r>0\colon r B_{X^\ast}\subseteq s_\e B_{X^\ast}\bigr\}
$$
and
$$
R_X(\e)=\inf\bigl\{R>0\colon s_\e B_{X^\ast}\subseteq RB_{X^\ast}\bigr\}.
$$
Since $(1-\tfrac{\e}{2})B_{X^\ast}\subseteq s_\e B_{X^\ast}$ for any infinite-dimensional Banach space $X$ (see \cite[Thm.~4.13]{CDDK2}), we always have $1-\tfrac{\e}{2}\leq r_X(\e)\leq R_X(\e)\leq 1$. In \cite{KM}, certain estimates on $r_X(\e)$ and $R_X(\e)$ were proved for Banach spaces with shrinking FDD's. The present work contains also some further development in this direction, in particular, for the Tsirelson, Schlumprecht and Baernstein spaces.  

Notice that whenever $s_\e B_{X^\ast}$ is a~ball, it must be a~ball centered at the origin; this follows immediately from Lemma~\ref{radial_L} below. Observe also that if $X$ is an Asplund space with $R_X(\e)=1$ for some $\e\in (0,2)$, then $s_\e B_{X^\ast}$ cannot be a~ball. For example, this happens for all $\e\in (0,2)$ in the case of Tsirelson's space (Proposition~\ref{T}) and $C(K)$-spaces with countably infinite compact Hausdorff space $K$ (\cite[Remark~4]{KM}). However, as it is witnessed by the dual $\mathcal{S}^\ast$ of Schlumprecht's space, the equation $R_X(\e)=1$ may hold true for some, but not all values of $\e\in (0,2)$; see Proposition~\ref{sch_P}. Not surprisingly, it is yet possible that $s_\e B_{X^\ast}$ is not a~ball, whereas $R_X(\e)<1$, but showing this requires some additional effort.

The following theorem summarizes the main results of the paper.
\begin{theoremA}
    For any separable Banach space $X$ the following assertions hold true:

    \vspace*{1mm}
    \begin{enumerate}[label={\rm (\roman*)},leftmargin=24pt]
    \setlength{\itemindent}{0mm}
    \setlength{\itemsep}{5pt}
    \item If X has property $(M^\ast)$, then the derivation $s_\e B_{X^\ast}$ is a~ball for every $\e\in (0,2)$.
    \item If $X$ has a shrinking FDD and satisfies property $(m_q^\ast)$ for some $q\in [1,\infty)$, then for every $\e\in (0,2)$, we have $s_\e B_{X^\ast}=(1-(\frac{\e}{2})^q)^{1/q}B_{X^\ast}$.
    \item The condition that $s_\e B_{X^\ast}$ is a ball for every $\e\in (0,2)$ does not imply that $X$ has property $(M^\ast)$; an~example is given by the dual of the Baernstein space $\mathcal{B}^\ast$ which fails to have property $(M^\ast)$, yet $s_\e B_{\mathcal{B}}=(1-(\tfrac{\e}{2})^2)^{1/2}B_{\mathcal{B}}$ for every $\e\in (0,2)$.
    \item If $X=\mathcal{T}$ is the original Tsirelson space, or $X=\mathcal{S}^\ast$ is the dual of the Schlumprecht space, then $s_\e B_{X^\ast}$ is not a~ball, for any $\e\in (0,2)$. More precisely:
        
        \vspace*{2mm}
        \begin{itemize}[leftmargin=24pt]
        \setlength{\itemsep}{5pt}
        \item $r_{\mathcal{T}}(\e)=1-\frac{\e}{4}$, $R_{\mathcal{T}}(\e)=1$;
    
        \item $r_{\mathcal{S}^\ast}(\e)\leq\inf\Big\{\frac{\log_2 (n+2)-\tfrac{\e}{2}}{\log_2 (n+1)}\colon n\in\N \Big\}$, $R_{\mathcal{S}^\ast}(\e)=\min\big\{1,\log_23-\frac{\e}{2}\big\}$.
        \end{itemize}
    \item If $X=\ell_{\Ma_{A,B}}^\ast$ is the dual of an~Orlicz sequence space induced by a~function of the form $\Ma_{A,B}(t)=At^4+Bt^2$,     then $s_\e B_{X^\ast}$ is not a~ball, for any $\e\in (0,2)$.
    \end{enumerate}
\end{theoremA}

\section{$(M^\ast)$-type properties}
Let $X$ be a~separable Banach space and $p\in [1,\infty]$. Below we recall the following important notions investigated by Kalton \cite{kalton}, and Kalton and Werner in \cite{KW}: 
\begin{enumerate}[label=(\alph*),leftmargin=24pt]
\setlength{\itemindent}{0mm}
\setlength{\itemsep}{1pt}
\item We say that $X$ has property ($m_p$), provided that for every $x\in X$ and any weakly null sequence $(x_n)_{n=1}^\infty\subset X$, we have
$$
\limsup_{n\to\infty}\n{x+x_n}=\n{(\n{x},\,\limsup_{n\to\infty}\n{x_n})}_p.
$$
\item We say that $X$ has property ($m_p^\ast$), provided that for every $x^\ast\in X^\ast$ and any weak$^\ast$-null sequence $(x_n^\ast)_{n=1}^\infty\subset X^\ast$, we have
$$
\limsup_{n\to\infty}\n{x^\ast+x_n^\ast}=\n{(\n{x^\ast},\, \limsup_{n\to\infty}\n{x_n^\ast})}_p.
$$
\end{enumerate}
More general versions of properties ($m_p$) and ($m_p^\ast$) are defined as follows:
\begin{enumerate}[label=(\alph*),leftmargin=24pt]
\setcounter{enumi}{2}
\setlength{\itemindent}{0mm}
\setlength{\itemsep}{1pt}
\item We say that $X$ has property ($M$), provided that for
 all $x,y\in X$ with $\n{x}=\n{y}$, and every weakly null sequence $(u_n)_{n=1}^\infty\subset X$, we have
$$
\limsup_{n\to\infty}\n{x+u_n}=\limsup_{n\to\infty}\n{y+u_n}.
$$
\item We say that $X$ has property ($M^\ast$), provided that for all $x^\ast,y^\ast\in X^\ast$ with $\n{x^\ast}=\n{y^\ast}$, and every weak$^\ast$-null sequence $(u_n^\ast)_{n=1}^\infty\subset X^\ast$, we have
$$
\limsup_{n\to\infty}\n{x^\ast+u_n^\ast}=\limsup_{n\to\infty}\n{y^\ast+u_n^\ast}.
$$
\end{enumerate}

It follows from \cite[Thm.~2.6]{KW} that if, for example, $X$ has separable dual and satisfies property $(M)$, then it also satisfies $(M^\ast)$. As we will now show, for those Banach spaces, the derivations $s_\e B_{X^\ast}$, for any $\e\in (0,2)$, are balls centered at the origin. We first need the following simple lemma.
\begin{lemma}\label{radial_L}
Let $X$ be a Banach space and $\e\in (0,2)$. If $x_0^\ast\in s_\e B_{X^\ast}$, then $\theta x_0^\ast\in s_\e B_{X^\ast}$ for every $\theta \in (0,1)$.
\end{lemma}
\begin{proof}
Fix any basic weak$^\ast$-open neighborhood $U$ of $\theta x_0^\ast$, say
    $$
    U=U_{\theta x_0^\ast}(x_1,\ldots,x_k;\eta)=\big\{y^\ast\in X^\ast\colon \abs{y^\ast(x_i)-\theta x_0^\ast(x_i)}<\eta   \mbox{ for }i=1,\ldots,k\big\}.
    $$
    Consider the weak$^\ast$-open neighborhood $V=V_{x_0^\ast}(x_1,\ldots,x_k;\eta)$ defined as above with $x_0^\ast$ in the place of $\theta x_0^\ast$. Since $x_0^\ast\in s_\e B_{X^\ast}$, there exist $u^\ast, v^\ast\in V\cap B_{X^\ast}$ such that $\n{u^\ast-v^\ast}>\e$. Define
    $$
    a^\ast=\frac{1+\theta}{2}u^\ast-\frac{1-\theta}{2}v^\ast\quad\mbox{ and }\quad b^\ast=\frac{1+\theta}{2}v^\ast-\frac{1-\theta}{2}u^\ast.
    $$
    Plainly, $\n{a^\ast}, \n{b^\ast}\leq \tfrac{1}{2}(1+\theta)+\tfrac{1}{2}(1-\theta)=1$. Also, for each $i=1,\ldots,k$, we have
    \begin{equation*}
        \begin{split}
        \abs{a^\ast(x_i)-\theta x_0^\ast(x_i)} &=\Bigg|\frac{1+\theta}{2}(u^\ast(x_i)-x_0^\ast(x_i))-\frac{1-\theta}{2}(v^\ast(x_i)-x_0^\ast(x_i))\Bigg|\\
        &<\Big(\frac{1+\theta}{2}+\frac{1-\theta}{2}  \Big)\eta=\eta
        \end{split}
    \end{equation*}
    and, similarly, $\abs{b^\ast(x_i)-\theta x_0^\ast(x_i)}<\eta$. Hence, $a^\ast, b^\ast\in U\cap B_{X^\ast}$. Finally, notice that $\n{a^\ast-b^\ast}=\n{u^\ast-v^\ast}>\e$ which shows that $\mathrm{diam} (U\cap B_{X^\ast})>\e$, therefore $\theta x_0^\ast\in s_\e B_{X^\ast}$.    
\end{proof}
\begin{proposition}
If $X$ is a separable Banach space with property $(M^\ast)$, then for every $\e\in (0,2)$, the $\e$-Szlenk derivation $s_\e B_{X^\ast}$ is a~ball centered at the origin.
\end{proposition}
\begin{proof}
Suppose that $\n{x_0^\ast}=\n{y_0^\ast}=r\leq 1$ and $x_0^\ast\in s_\e B_{X^\ast}$. Pick a~sequence $(x_n^\ast)_{n=1}^\infty\subset B_{X^\ast}$ with $w^\ast$-$\lim_n x_n^\ast=x_0^\ast$ and such that for each $k\in\N$ there exist $m,n\geq k$ with $\n{x_m^\ast-x_n^\ast}>\e$. Let $u_n^\ast=x_n^\ast-x_0^\ast$ for $n\in\N$; then $(y_0^\ast+u_n^\ast)_{n=1}^\infty$ is weak$^\ast$ convergent to $y_0^\ast$ and property ($M^\ast$) implies that
\begin{equation}\label{sup}
\limsup_{n\to\infty}\n{y_0^\ast+u_n^\ast}=\limsup_{n\to\infty}\n{x_0^\ast+u_n^\ast}=\limsup_{n\to\infty}\n{x_n^\ast}\leq 1.
\end{equation}
Observe also that for any $k\in\N$ there are $m,n\geq k$ such that 
$$
\n{(y_0^\ast+u_m^\ast)-(y_0^\ast+u_n^\ast)}=\n{x_m^\ast-x_n^\ast}>\e.
$$
Hence, modifying slightly each $u_n^\ast$ and using \eqref{sup} we can guarantee that the last condition is still true whereas $\n{y_0^\ast+u_n^\ast}\leq 1$ for each $n\in\N$. This shows that $y_0^\ast\in s_\e B_{X^\ast}$ and, in combination with Lemma~\ref{radial_L}, we conclude that $rB_{X^\ast}\subset s_\e B_{X^\ast}$ which completes the proof.
\end{proof}

Now, we observe that in the case where the stronger condition, namely $(m_q^\ast)$ with some $q\in [1,\infty)$, is satisfied, then the ball $s_\e B_{X^\ast}$ can be explicitly described. This generalizes the corresponding result for $\ell_p$-sums of finite-dimensional spaces (see \cite[Cor.~5]{KM}).

\begin{proposition}\label{m_q_der_P}
    If $X$ is a Banach space with a shrinking FDD and with property {\rm (}$m_q^\ast${\rm )}, for some $q\in [1,\infty)$, then for every $\e\in (0,2)$, we have 
    $$
        s_\e B_{X^\ast}=\Big(1-\Big(\frac{\e}{2}\Big)^q\Big)^{1/q} B_{X^\ast}.
    $$
\end{proposition}
\begin{proof}
Fix any $\e\in (0,2)$ and suppose that there exists $x_0^\ast\in s_\e B_{X^\ast}$ with 
$$
\n{x_0^\ast}>\varrho(\e)\coloneqq \Big(1-\Big(\frac{\e}{2}\Big)^q\Big)^{1/q}.
$$
Then, there is a sequence $(x_n^\ast)_{n=1}^\infty\subset B_{X^\ast}$ satisfying
$$
w^\ast\mbox{-}\!\lim_{n\to\infty} x_n^\ast=x_0^\ast\quad\mbox{ and }\quad \n{x_n^\ast-x_0^\ast}>\frac{\e}{2}\quad\mbox{for each }\, n\in\N.
$$
The property $(m_q^\ast)$ applied to the weak$^\ast$-null sequence $(u_n^\ast)_{n=1}^\infty$ defined by $u_n^\ast=x_n^\ast-x_0^\ast$ yields
\begin{equation*} \label{mq*}
\limsup_{n\to\infty} \n{x_0^\ast+u_n^\ast}= \n{\ \n{x_0^\ast},\limsup_{n\to\infty}\n{u_n^\ast} \ }_q, 
\end{equation*}
that is, with the notation $\eta=\n{x_0^\ast}^q- (\varrho(\e))^q>0$, we have
\begin{equation*}
\begin{split}
\limsup_{n\to\infty} \n{x_n^\ast}^q &=\n{x_0^\ast}^q+\limsup_{n\to\infty}\n{x_n^\ast-x_0^\ast}^q\\
&= (\varrho(\e))^q+\eta+\limsup_{n\to\infty}\n{x_n^\ast-x_0^\ast}^q\\
&\geq (\varrho(\e))^q+\eta+\Big(\frac{\e}{2}\Big)^q=1+\eta.
\end{split}
\end{equation*}
This plainly contradicts the fact that $\n{x_n^\ast}\leq 1$ for each $n\in\N$. 

We have shown that $s_\e B_{X^\ast}\subseteq \varrho(\e) B_{X^\ast}$. For the converse inclusion, let $(E_n)_{n=1}^\infty$ be a~shrinking FDD of $X$, $(E_n^\ast)_{n=1}^\infty$ be its dual FDD, and for every $n\in\N$ pick $e_n^\ast\in E_n^\ast$ with $\n{e_n^\ast}=1$. Fix any $\e\in (0,2)$ and any $x_0^\ast\in X^\ast$ satisfying
\begin{equation}\label{x_0_in_span}
x_0^\ast\in\mathrm{span}\bigcup_{n=1}^\infty E_n^\ast\quad\mbox{ and }\quad \n{x_0^\ast}<\varrho(\e),
\end{equation}
say, $x_0^\ast=\sum_{i=1}^N x_i^\ast$ with $x_i^\ast\in E_i^\ast$ for $1\leq i\leq N$. Pick $\e^\prime\in (\e,2)$ so that $\n{x_0^\ast}<\varrho(\e^\prime)<\varrho(\e)$ and define 
$$
x_{n,\pm}^\ast=x_0^\ast\pm\frac{\e^\prime}{2}e_{N+n}^\ast \quad\, (n\in\N).
$$
Clearly, the sequences $(x_{n,+}^\ast)_{n=1}^\infty$ and $(x_{n,-}^\ast)_{n=1}^\infty$ are weak$^\ast$-convergent to $x_0^\ast$. On the other hand, property $(m_q^\ast)$ implies that 
\begin{equation*}
\limsup_{n\to\infty}\n{x_{n,\pm}^\ast}^q =  \n{x_0^\ast}^q+\Big(\frac{\e^\prime}{2}\Big)^{\!q} \limsup_{n\to\infty}\n{e_{N+n}^\ast}^q<\varrho(\e^\prime)^q+\Big(\frac{\e^\prime}{2}\Big)^{\!q}=1.
\end{equation*}
Hence, $\n{x_{n,\pm}^\ast}\leq 1$ for sufficiently large $n\in\N$. Moreover, $\n{x_{n,+}^\ast-x_{n,-}^\ast}=\e^\prime>\e$ for each $n\in\N$, which shows that $x_0^\ast\in s_\e B_{X^\ast}$.

Since every $x_0^\ast$ satisfying \eqref{x_0_in_span} belongs to $s_\e B_{X^\ast}$, and in view of the fact that this derivation is norm closed, we infer that $\varrho(\e)B_{X^\ast}\subseteq s_\e B_{X^\ast}$ which completes the proof.    
\end{proof}


\section{Derivations of Tsirelson's and Schlumprecht's spaces}

We use standard notation: by $c_{00}$ we mean the vector space of real sequences that vanish except a~finite set; for any real sequence $x=(x_n)_{n=1}^\infty$ we denote $\mathrm{supp}(x)=\{n\in\N\colon x_n\neq 0\}$. For any finite sets $E,F\subset\N$, we write $E<F$ provided $\max E<\min F$, and we write $k\leq E$ if $k\leq\min E$. In the case where the canonical unit vectors $(e_n)_{n=1}^\infty$ form a~basis of a~Banach space $X$, we denote by $P_n\colon X\to\mathrm{span} \{e_j\colon 1\leq j\leq n\}$ the canonical projection, i.e. $P_n(\sum_{j=1}^\infty a_je_j)=\sum_{j=1}^n a_je_j$.

Following \cite{CS}, we denote by $T$ the Figiel--Johnson space introduced in \cite{figiel_johnson}, which is the completion of $c_{00}$ with respect to the norm given by the implicit formula
$$
\n{\xi}_{T}=\max\Biggl\{\n{\xi}_\infty,\sup_{k\leq E_1<\ldots<E_k}\frac{1}{2}\sum_{j=1}^k\n{E_j\xi}_{T}\Biggr\}.
$$
It is well-known to be a~reflexive Banach space with no isomorphic copy of $\ell_p$, for any $p\in (1,\infty)$. The space $T$ is the dual of the original Tsirelson's space introduced in \cite{tsirelson}, which we denote by $\mathcal{T}$, so that $\mathcal{T}=T^\ast$. As we will now see, the $\e$-Szlenk derivations $s_\e B_{T}$ are far away from being balls.
\begin{proposition}\label{T}
For every $\e\in (0,2)$, we have 
$$
r_{\mathcal{T}}(\e)=1-\frac{\e}{4}\quad\mbox{ and }\quad R_{\mathcal{T}}(\e)=1.
$$
\end{proposition}
\begin{proof}
It is easy to show that $R_\Ta(\e)=1$. Indeed, since $w^\ast$-$\lim_n (e_1^\ast\pm e_n^\ast)=e_1^\ast$ and $\n{e_1^\ast\pm e_n^\ast}=1$ for each $n\geq 2$, for every $w^\ast$-open neighborhood $U$ of $e_1^\ast$, we have $e_1^\ast\pm e_n^\ast\in U\cap B_T$ for sufficiently large $n\in\N$. Hence, $\mathrm{diam} (U\cap B_T)\geq 2\n{e_n^\ast}=2$ which shows that $e_1^\ast\in s_\e B_T$ for every $\e\in (0,2)$.

In order to show that $r_\Ta(\e)\leq 1-\tfrac{\e}{4}$, fix any indices $3\leq m<n$, any $r\in (1-\tfrac{\e}{4},1]$, and consider
$$
x_0^\ast=r(e_m^\ast+e_n^\ast).
$$
Pick any $\e^\prime\in (0,\e)$ such that $\n{x_0^\ast}=r>1-\tfrac{\e^\prime}{4}$. Take $0<\delta<\tfrac{1}{6n}(\e-\e^\prime)$ and define a~$w^\ast$-open neighborhood of $x_0^\ast$ by
$$
V=\big\{y^\ast\in T\colon \abs{y^\ast(e_i)-x_0^\ast(e_i)}<\delta\mbox{ for each }1\leq i\leq n\big\}.
$$
Since the norm in $T$ is majorized by the $\ell_1$-norm, we have $\n{x_0^\ast-P_n(y^\ast)}\leq n\delta$ whenever $y^\ast\in V$. Hence, if $y^\ast\in V\cap B_T\cap c_{00}$, then
\begin{equation*}
\begin{split}
1 & \geq \n{y^\ast}=\n{P_n(y^\ast)+(I-P_n)(y^\ast)}\\[4pt]
& \geq \n{x_0^\ast+(I-P_n)(y^\ast)}-n\delta\\
& \geq \frac{1}{2}(2r+\n{(I-P_n)(y^\ast)})-n\delta,
\end{split}
\end{equation*}
where the last inequality is obtained by using the admissible collection $\{m\}<\{n\}<\supp (I-P_n)(y^\ast)$. Hence, for any $y^\ast\in V\cap B_T$, we have
$$
\n{(I-P_n)(y^\ast)}\leq 2(1-r+n\delta)<\frac{\e^\prime}{2}+2n\delta.
$$
Therefore, for arbitrary $y^\ast, z^\ast\in V\cap B_T$, we have $\n{y^\ast-z^\ast}<\e^\prime+6n\delta<\e$, which shows that $\mathrm{diam} (V\cap B_T)\leq\e$. We have thus shown that for every $r>1-\tfrac{\e}{4}$ there is $x_0^\ast\in T$ with $\n{x_0^\ast}=r$ and such that $x_0^\ast\not\in s_\e B_T$. This proves the inequality $r_\Ta(\e)\leq 1-\tfrac{\e}{4}$.

Finally, we must show that $r_\Ta(\e)\geq 1-\tfrac{\e}{4}$. Since $s_\e B_T$ is a closed set, it suffices to prove that every $x_0^\ast\in T$ with $\n{x_0^\ast}<1-\tfrac{\e}{4}$ belongs to $s_\e B_T$. Since $c_{00}$ is dense in $T$, we may assume that $x_0^\ast$ has finite support, and let $N=\max\,\supp (x_0^\ast)$. Pick $\e^\prime\in (\e,2)$ such that $\n{x_0^\ast}<1-\tfrac{\e^\prime}{4}$. For any fixed pair $(k,m)$ of natural numbers with $2\leq m\leq N+k$, we define
$$
x_{k,m,\pm}^\ast=x_0^\ast\pm \alpha_{k,m}\big(e_{N+k+1}^\ast+\ldots+e_{N+k+m}^\ast\big),
$$
where
$$
\alpha_{k,m}=\frac{\e^\prime}{2N+m}.
$$
Notice that the inequalities imposed on $k$ and $m$ imply
\begin{equation}\label{e_norm}
\n{e_{N+k+1}^\ast+\ldots+e_{N+k+m}^\ast}=\frac{1}{2}m.
\end{equation}
We are going to show that
\begin{equation}\label{x_norm}
\n{x_{k,m,\pm}^\ast}\leq\max\Bigg\{1,\,\, \frac{1}{2}m\alpha_{k,m},\,\, \frac{1}{4}(2N+m)\alpha_{k,m}\Bigg\}.
\end{equation}
To this end, let $F_{k,m}=\{N+k+1,\ldots,N+k+m\}$ and fix any admissible collection $E_1<\ldots<E_r$. Without loss of generality we may assume that $\ccup_{i=1}^r E_i\subseteq\supp
 (x_{k,m,\pm}^\ast)=\supp (x_0^\ast)\cup F_{k,m}$. We consider two possible cases.

\vspace*{2mm}\noindent
\underline{{\it Case 1.}}\, $\supp(x_0^\ast)<E_1$. Then, in view of \eqref{e_norm}, we have
$$
\frac{1}{2}\sum_{i=1}^r\n{E_i x_{k,m,\pm}^\ast}\leq\alpha_{k,m}\n{e_{N+k+1}^\ast+\ldots+e_{N+k+m}^\ast}=\frac{1}{2}m\alpha_{k,m}.
$$
\vspace*{2mm}\noindent
\underline{{\it Case 2.}}\, $E_1\cap\supp(x_0^\ast)\not=\varnothing$. This means that $\min E_1\leq N$, which by the admissibility condition implies that $r\leq N$. Notice that for each $i$, we have
$$
\n{E_i(e_{N+k+1}^\ast+\ldots+e_{N+k+m}^\ast)}=\left\{\begin{array}{cl} 0 & \mbox{if }E_i\cap F_{k,m}=\varnothing\\
1 & \mbox{if }\abs{E_i\cap F_{k,m}}=1\\
\tfrac{1}{2}\abs{E_i\cap F_{k,m}} & \mbox{if }\abs{E_i\cap F_{k,m}}\geq 2.\end{array}\right.
$$
Define $I=\{i\leq r\colon \abs{E_i\cap F_{k,m}}=1\}$ and $J=\{i\leq r\colon\abs{E_i\cap F_{k,m}}\geq 2\}$. We have
\begin{equation*}
\begin{split}
\frac{1}{2}\sum_{i=1}^r\n{E_i x_{k,m,\pm}^\ast} & \leq \frac{1}{2}\sum_{i=1}^r\n{E_i x_0^\ast}+\frac{1}{2}\alpha_{k,m}\sum_{i=1}^r\n{E_i(e_{N+k+1}^\ast+\ldots+e_{N+k+m}^\ast)}\\
& \leq\n{x_0^\ast}+\frac{1}{2}\alpha_{k,m}\Big(\abs{I}+\frac{1}{2}\sum_{j\in J}\abs{E_j\cap F_{k,m}}\Big)\\
& \leq\n{x_0^\ast}+\frac{1}{2}\alpha_{k,m}\Big(N+\frac{1}{2}m\Big).
\end{split}
\end{equation*}
Thus, inequality \eqref{x_norm} has been proved. Since $\n{x_0^\ast}<1-\tfrac{\e}{4}$ and
$$
\alpha_{k,m}=\frac{\e^\prime}{2N+m}<\frac{2}{m},
$$
we see that \eqref{x_norm} implies $\n{x_{k,m,\pm}^\ast}\leq 1$.
Now, if for any $m\geq 2$ we take $k=m$, and let $m\to\infty$, then we see that the sequences $(x_{m,m,+}^\ast)_{m=2}^\infty$ and $(x_{m,m,-}^\ast)_{m=2}^\infty$ are $w^\ast$-convergent to $x_0^\ast$ and for each $m\geq 2$, we have
\begin{equation*}
\n{x_{m,m,+}^\ast-x_{m,m,-}^\ast}=2\alpha_{m,m}\n{e_{N+m+1}^\ast+\ldots+e_{N+2m}^\ast}=\frac{m\e^\prime}{2N+m}\xrightarrow[\,m\to\infty\,]{}\e^\prime>\e.
\end{equation*}
This shows that for every $w^\ast$-open neighborhood $V$ of $x_0^\ast$, we have $\mathrm{diam} (V\cap B_T)>\e$ which shows that $x_0^\ast\in s_\e B_T$.
\end{proof}

Recall that the Schlumprecht space $\mathcal{S}$, introduced in \cite{schlumprecht}, is the completion of $c_{00}$ with respect to the norm given by the implicit formula
$$
\n{\xi}_{\mathcal{S}}=\max\Biggl\{\n{\xi}_\infty,\sup_{E_1<\ldots<E_n}\frac{1}{\phi(n)}\sum_{j=1}^n\n{E_j\xi}_{\mathcal{S}}\Biggr\},
$$
where $\phi(n)=\mathrm{log}_2(n+1)$. The space $\mathcal{S}$ is reflexive (see \cite[Thm.~I.10]{AT}) and the canonical basis $(e_n)_{n=1}^\infty$ is $1$-subsymmetric and $1$-unconditional (see \cite[Prop.~2]{schlumprecht}). Below, we estimate Szlenk derivations of the dual unit ball $B_{\mathcal{S}}$ of $\mathcal{S}^\ast$.

\begin{lemma}[{\cite[Lemma 4]{schlumprecht}}]\label{sch_L}
For every $n\in\N$, we have
$$
\Bigg\|\sum_{j=1}^n e_j\Bigg\|_{\mathcal{S}}=\frac{n}{\phi(n)}.
$$
\end{lemma}
 
\begin{proposition}\label{sch_P}
For every $\e\in (0,2)$, we have 
\begin{equation}\label{r_schlumprecht}
r_{\mathcal{S}^\ast}(\e)\leq\inf\Bigg\{\frac{\log_2 (n+2)-\tfrac{\e}{2}}{\log_2 (n+1)}\colon n\in\N \Bigg\}
\end{equation}
and
\begin{equation}\label{R_schlumprecht}
R_{\mathcal{S}^\ast}(\e)=\min\Big\{1,\log_23-\frac{\e}{2}\Big\}.
\end{equation}
\end{proposition}
\begin{proof}
Fix any $\e\in (0,2)$ and define $N_\e=\{n\in\N\colon \phi(n+1)-\phi(n)<\tfrac{\e}{2}\}$; this is an infinite set as $\phi(n+1)-\phi(n)\to 0$. Fix $n\in N_\e$ and take any $r>0$ satisfying 
\begin{equation}\label{r_ineq}
\frac{\phi(n+1)-\tfrac{\e}{2}}{n}<r\leq \frac{\phi(n)}{n}.
\end{equation}
Define $x_0=r(e_1+\ldots+e_n)\in\mathcal{S}$. In view of Lemma~\ref{sch_L} and inequality \eqref{r_ineq}, we have $\n{x_0}\leq 1$. For any, temporarily fixed, $\delta>0$ define a~weak$^\ast$-open neighborhood of $x_0$ as
$$
V_\delta=\big\{y\in\mathcal{S}\colon \abs{e_i^\ast(y)-r}<\delta\mbox{ for each }1\leq i\leq n\big\}.
$$
For every $y\in V_\delta\cap B_{\mathcal{S}}\cap c_{00}$ we have
\begin{equation*}
\begin{split}
    1 &\geq \n{y}\geq \frac{1}{\phi(n+1)}\Big(\sum_{i=1}^n\abs{e_i^\ast(y)}+\n{(I-P_n)(y)}\Big)\\
    &>\frac{1}{\phi(n+1)}\big(nr-n\delta+\n{(I-P_n)(y)}
    \big).
\end{split}
\end{equation*}
Hence, $\n{(I-P_n)(y)}<\phi(n+1)-nr+n\delta$. In view of \eqref{r_ineq}, by decreasing $\delta$ if necessary, we may guarantee that there is a~weak$^\ast$-open neighborhood $V$ of $x_0$ such that $\n{(I-P_n)(y)}<\mu$ for every $y\in V\cap B_{\mathcal{S}}\cap c_{00}$, with some constant $\mu<\tfrac{\e}{2}$, and that $\n{y-z}\leq \tfrac{\e}{2}$ for all $y,z\in V\cap B_{\mathcal{S}}$. Therefore, $\mathrm{diam}(V\cap B_{\mathcal{S}})\leq\e$ which shows that $x_0\not\in s_\e B_{\mathcal{S}}$. Recall that $r$ could be chosen to be arbitrarily close to the left-hand side of \eqref{r_ineq}, whereas $\n{x_0}=nr/\phi(n)$, therefore
$$
r_{\mathcal{S}^\ast}(\e)\leq \frac{\phi(n+1)-\tfrac{\e}{2}}{\phi(n)}.
$$
Since $n\in N_\e$ was arbitrary, inequality \eqref{r_schlumprecht} follows.

In order to show equation \eqref{R_schlumprecht}, assume first that $\e\leq 2(\phi(2)-1)$, that is, the minimum in \eqref{R_schlumprecht} equals $1$. Then for any $m,n\in\N$, $m<n$, we have 
$$
\Big\|e_m\pm\frac{\e}{2}e_n\Big\|=\max\Big\{1,\,\frac{2+\e}{2\phi(2)}\Big\}=1.
$$
Since $w^\ast$-$\lim_n (e_m\pm \tfrac{\e}{2}e_n)=e_m$, we have $e_m\in s_\e B_{\mathcal{S}}$ which shows that in this case $R_{\mathcal{S}^\ast}(\e)=1$. 

Now, assume that $2(\phi(2)-1)<\e<2$ and fix any $x_0\in \mathcal{S}$ with $\phi(2)-\tfrac{\e}{2}<\n{x_0}\leq 1$. Pick $\e^\prime \in (2(\phi(2)-1),\e)$ and large enough $N\in\N$ so that $\n{P_N x_0}>\phi(2)-\tfrac{\e^\prime}{2}$. Take $0<\delta< \tfrac{1}{6N}(\e-\e^\prime)$ and define a~weak$^\ast$-open neighborhood of $x_0$ by
$$
V=\big\{y\in\mathcal{S}\colon \abs{e_i^\ast(y-x_0)}<\delta\mbox{ for each }1\leq i\leq N\big\}.
$$
Notice that for $y\in V$, we have $\n{P_Ny}\geq \n{P_Nx_0}-N\delta$. Hence, if $y\in V\cap B_{\mathcal{S}}\cap c_{00}$, then 
\begin{equation*}
\begin{split}
    1 &\geq \n{y}\geq \frac{1}{\phi(2)}\big(\n{P_N(y)}+\n{(I-P_N)(y)}\big)\\
    &\geq\frac{1}{\phi(2)}\big(\n{P_N x_0}-N\delta+\n{(I-P_N)(y)}\big).
\end{split}
\end{equation*}
Therefore, for every $y\in V\cap B_{\mathcal{S}}$, we have
$$
\n{(I-P_N)(y)}\leq \phi(2)-\n{P_N x_0}+N\delta<\frac{\e^\prime}{2}+N\delta.
$$
Consequently, for all $y,z\in V\cap B_{\mathcal{S}}$, we have $\n{y-z}\leq \e^\prime+6N\delta<\e$. This shows that $x_0\not\in s_\e B_{\mathcal{S}}$ and, therefore, $R_{\mathcal{S}^\ast}(\e)\leq \phi(2)-\tfrac{\e}{2}$.

For the converse inequality, fix any $0<r<\phi(2)-\tfrac{\e}{2}$ and set $x_0=re_1$. Obviously, $w^\ast$-$\lim_n (x_0+\tfrac{\e}{2}e_n)=x_0$ and for any $n\geq 2$, we have 
$$
\n{x_0+\frac{\e}{2}e_n}=\max\Big\{r,\frac{\e}{2},\frac{r+\tfrac{\e}{2}}{\phi(2)}\Big\}<1,
$$
which shows that $x_0\in s_\e B_{\mathcal{S}}$ and completes the proof.
\end{proof}

\vspace*{1mm}\noindent
{\bf Open problem. }We do not know if the inequality converse to \eqref{r_schlumprecht} holds true.


\section{Derivations of Baernstein's space}
From now on, we denote by $\mathcal{B}$ the Baernstein space introduced in \cite{baernstein} (see also \cite[Ch.~0]{CS}), which is a~reflexive version of the classical Schreirer space. The space $\Bb$ is defined as the completion of $c_{00}$ with respect to the norm
$$
\begin{array}{r}
\n{x}_\mathcal{B}=\sup\Biggl\{  \Biggl(\displaystyle{\sum_{i=1}^{n}\n{E_i x}^2_{\ell_1}\biggr)^{\!1/2}}\!\colon n\in\N,\, E_1<\ldots <E_n\subset\N ,\,\,\qquad\qquad \\[-10pt]
\abs{E_i}\leq\min E_i\,\mbox{ for each }\,i=1,\dots ,n\Biggr\}.
\end{array}
$$ 
Alternatively, $\Bb$ can be defined as the~space of all real sequences $(x_n)_{n=1}^\infty$ such that 
$$
\begin{array}{r}
\sup\Biggl\{  \Biggl(\displaystyle{\sum_{i=1}^{\infty}\n{E_i x}^2_{\ell_1}\biggr)^{\!1/2}}\!\colon E_1<E_2<\ldots\subset\N ,\,\,\hspace*{110pt} \\[-10pt]
\abs{E_i}\leq\min E_i\,\mbox{ for each }\, i=1,2,\dots\Biggr\}<\infty
\end{array}
$$
with the expression on the left-hand side defining the norm in $\Bb$. It is known that this space is reflexive and does not have the Banach--Saks property, although its dual $\Bb^\ast$ does, as was shown by Seifert \cite{seifert}. 

In what follows, any finite set $E\subset\N$ with $\abs{E}\leq\min E$ will be called {\it admissible}. Let $(e_n)_{n=1}^\infty$ be the standard unit vector basis of $\Bb$; it is easily seen that this basis is monotone and unconditional. For $n\in\N$, let $P_n\colon \Bb\to\Bb$ be the canonical projection onto the subspace $\mathrm{span}\,\{e_j\colon j\leq n\}$. We will need the following estimate observed by Partington in the proof of his \cite[Thm.~3]{partington}, and displayed by Bana\'s, Olszowy and Sadarangani \cite{BOS} in order to derive a~formula for the modulus of near convexity of $\Bb$.

\begin{lemma}\label{partlemma}  
For any $x\in\mathcal{B}$ and $n\in\N$, we have
$$
\n{x}_\mathcal{B}^2\geq \n{P_n x}^2_\mathcal{B}+ \n{(I-P_n)x}^2_\mathcal{B}.
$$
\end{lemma}

\vspace*{1mm}
The next lemma follows immediately from the inequality $\n{\cdot}_2\leq \n{\cdot}_1$, and the fact that the set $E=\{n,n+1,\ldots,2n-1\}$ is admissible for every $n\in\N$.
\begin{lemma}\label{easy_L}
For every $n\in\N$, we have
$$
\Bigg\|\sum_{i=n}^{2n-1} e_i\Bigg\|_\Bb=\Bigg\|\sum_{i=n}^{2n-1} e_i\Bigg\|_{\ell_1}\!\!=n.
$$
\end{lemma}

\begin{proposition}
For every $\e\in(0,2)$, we have 
$$s_\e B_{\mathcal{B}^{\dast}}= \Big(1-\Big(\frac{\e}{2}\Big)^{\! 2}\,\Big)^{\! 1/2} B_{\mathcal{B}^{\dast}}.$$
\end{proposition}
\begin{proof}
First, we observe that by combining Lemma~\ref{partlemma} and \cite[Thm.~2]{KM}, we obtain
$$
R_{\mathcal{B}^\ast}(\e)\leq \Big(1-\Big(\frac{\e}{2}\Big)^{\! 2}\,\Big)^{\! 1/2}\qquad (0<\e<2).
$$
Indeed, the assumptions of \cite[Thm.~2]{KM} are satisfied with $Z=\Bb^\ast$, $C_1=C_2=1$ and $\varphi(t)=\psi(t)=\chi(t)=t^2$. Therefore, it suffices to show that 
\begin{equation}\label{rBw}
r_\mathcal{B^\ast}(\e)\geq \Big(1-\Big(\frac{\e}{2}\Big)^{\! 2}\,\Big)^{\! 1/2}\qquad\, (0<\e<2).
\end{equation}

As $\Bb$ is reflexive, we may identify functionals acting on $\Bb^\ast$ with vectors in $\Bb$. For any $\alpha\in (0,2)$, let $r(\alpha)=(1-(\tfrac{\alpha}{2})^2)^{1/2}$; fix any $\e\in (0,2)$ and any finitely supported vector $x_0\in\Bb$ with $\n{x_0}_\Bb<r(\e)$, say
$$
x_0=\sum_{i=1}^N e_i^\ast(x_0) e_i,
$$
where $(e_n^\ast)_{n=1}^\infty\subset\Bb^\ast$ is the sequence of biorthogonal functionals for the basis $(e_n)_{n=1}^\infty$. Pick also $\delta\in (\e,2)$ so that $\n{x_0}_\Bb<r(\delta)<r(\e)$. We define two sequences $(x_{n,+})_{n=1}^\infty$ and $(x_{n,-})_{n=1}^\infty$ in $\Bb$ by the formula
\begin{equation}\label{x_n_def}
x_{n,\pm}=x_0\pm\frac{\delta}{2n} \sum_{i=n}^{2n-1}e_i \qquad (n\in\N).
\end{equation}

For each $n\in\N$, take any sequence $E_{1,n}<E_{2,n}<\ldots<E_{\nu(n),n}$ of admissible sets such that
\begin{equation}\label{x_n_adm}
\n{x_{n,\pm}}^2=\sum_{i=1}^{\nu(n)}\n{E_{i,n}x_{n,\pm}}^2_{\ell_1}.
\end{equation}
Fix any $n\geq N+1$ and observe that there is at most one index $i\in\{1,\ldots,\nu(n)\}$ satisfying $E_{i,n}x_0\neq 0$ and $E_{i,n}(\sum_{j=n}^{2n-1} e_j)\neq 0$. Define $i(n)$ as the largest number $i\in \{1,\ldots,\nu(n)\}$ with $\min E_{i,n}\leq N$, i.e. $E_{i,n}\cap \mathrm{supp} (x_0)\neq\varnothing$. Then, using \eqref{x_n_def}, \eqref{x_n_adm} and Lemma~\ref{easy_L}, we obtain

\begin{equation*}
    \begin{split}
        \n{x_{n,\pm}}_\Bb^2 & = \sum_{i=1}^{\nu(n)} \Bigg(\n{E_{i,n}x_0}_{\ell_1}+\frac{\delta}{2n}\Big\|E_{i,n}\Big(\sum_{j=n}^{2n-1} e_j\Big)\Big\|_{\ell_1} \Bigg)^{\!\! 2}\\
        & = \sum_{i=1}^{\nu(n)} \n{E_{i,n}x_0}_{\ell_1}^2+\frac{\delta}{n}\sum_{i=1}^{\nu(n)} 
        \n{E_{i,n}x_0}_{\ell_1}\!\cdot\!\Big\|E_{i,n}\Big(\sum_{j=n}^{2n-1}e_j\Big)\Big\|_{\ell_1}+\\
        & \hspace*{134pt} +\frac{{\delta}^2}{4n^2}\sum_{i=1}^{\nu(n)}\Big\|E_{i,n}\Big(\sum_{j=n}^{2n-1}e_j\Big)\Big\|_{\ell_1}^2\\
        & \leq \n{x_0}_\Bb^2+\frac{\delta}{n}\n{E_{i(n),n}x_0}_{\ell_1}\!\cdot\!\Big\|E_{i(n),n}\Big(\sum_{j=n}^{2n-1} e_j\Big)\Big\|_{\ell_1}+\frac{\delta^2}{4n^2}\Big\|\sum_{j=n}^{2n-1} e_j\Big\|_\Bb^2\\
        & \leq \n{x_0}_\Bb^2+\frac{\delta N^2}{n}+\frac{\delta^2}{4}.
    \end{split}
\end{equation*}
Hence, for the constant $\theta\coloneqq \n{x_0}_\Bb^2+\tfrac{\delta^2}{4}\in (0,1)$, we have the estimate $\n{x_{n,\pm}}^2\leq\theta+\tfrac{1}{n}\delta N^2$ for each $n>N$. Therefore, we may pick $n_0>N$ so that $\n{x_{n,\pm}}_\Bb\leq 1$ for every $n\geq n_0$. Moreover, by Lemma~\ref{easy_L}, we have $\n{x_{n,+}-x_{n,-}}_\Bb=\delta>\e$.

Since $\Bb$ is reflexive, $(e_n^\ast)_{n=1}^\infty$ is a~shrinking basis of $\Bb^\ast$ and hence $w^\ast$-$\lim_n e_n=0$. Therefore, $w^\ast$-$\lim_n x_{n,\pm}=x_0$ in $\Bb^\dast=\Bb$. We have thus shown that every weak$^\ast$-open neighborhood $U$ of $x_0$ satisfies $\mathrm{diam} (U\cap B_{\Bb})>\e$ which means that $x_0\in s_\e B_{\Bb^{\dast}}$. Since $x_0$ was an~arbitrary sequence of finite support satisfying $\n{x_0}<r(\e)$, and the derivation $s_\e B_{\Bb^\dast}$ is closed, we infer that $r(\e)B_{\Bb^\dast}\subset s_\e B_{\Bb^\dast}$ which yields \eqref{rBw} and completes the proof. 
\end{proof}

\begin{proposition}
The space $\mathcal{B}^\ast$ does not have  property {\rm (}$M^\ast${\rm )}.
\end{proposition}

\begin{proof}
Define
$$
x_0=e_1,\quad y_0=\frac{1}{\sqrt{2}}(e_1+e_2),\,\,\,\mbox{ and }\,\,\, u_n=e_n\xrightarrow[n\to\infty]{w\ast}0.
$$
Obviously, we have $\n{x_0}_\Bb=\n{y_0}_\Bb=1$. However,  
$$
\n{x_0+u_n}_\Bb=\n{e_1+e_n}_\Bb=\sqrt{2}\quad (n>1),
$$
whereas
$$
\n{y_0+u_n}_\Bb=\Big\|\frac{1}{\sqrt{2}}(e_1+e_2)+e_n\Big\|_\Bb=\sqrt{\Big(\frac{1}{\sqrt{2}}\Big)^{\! 2}+\Big(\frac{1}{\sqrt{2}}+1 \Big)^{\! 2}}=\sqrt{2+\sqrt{2}}\quad (n>2).
$$

\end{proof}



\section{Derivations of sequential Orlicz spaces}

In this section, we give a~general criterion which implies that Szlenk derivations are not balls, and we illustrate its application to a~class of some sequential Orlicz spaces.

\begin{lemma}\label{test_L}
    Let $X$ be a separable Banach space. Suppose there exist $x_0^\ast, y_0^\ast\in X^\ast$ with $\n{x_0^\ast}=\n{y_0^\ast}=r>0$ and continuous functions $\phi,\psi\colon (0,\infty)\to (0,\infty)$ with 
    $$
    \psi(\delta)<\min\{\delta,\phi(\delta)\}\qquad (\delta>0)
    $$
    such that for every $\delta>0$ the following conditions hold true:

\vspace*{1mm}
\begin{enumerate}[label={\rm (\roman*)},leftmargin=24pt]
\setlength{\itemindent}{0mm}
\setlength{\itemsep}{1pt}
\item there is a weak$^\ast$-null sequence $(v_n^\ast)_{n=1}^\infty\subset X^\ast$ with $\n{v_n^\ast}>\delta$ {\rm (}$n\in\N${\rm )} such that  
$$
\n{x_0^\ast\pm v_n^\ast}\leq r+\psi(\delta) \quad \mbox{for each }\,n\in\N;
$$

\item for every weak$^\ast$-null sequence $(u_n^\ast)\subset X^\ast$ with $\n{u_n^\ast}>\delta$ {\rm (}$n\in\N${\rm )}, we have
$$
\n{y_0^\ast+u_m^\ast}\geq r+\phi(\delta)\quad \mbox{for some }\, m\in\N.
$$
\end{enumerate}

\vspace*{1mm}\noindent
Then, for every $\e\in (0,2)$, the Szlenk derivation $s_\e B_{X^\ast}$ is not a~ball centered at the origin.
\end{lemma}

\begin{proof}
For any $\delta>0$, define 
$$
\e(\delta)=\frac{2\delta}{r+\psi(\delta)}.
$$
Notice that, by (i), we have $2\delta<2\n{v_n^\ast}\leq\n{x_0^\ast+v_n^\ast}+\n{x_0^\ast-v_n^\ast}$ which shows that at least one of the terms $\n{x_0^\ast\pm v_n^\ast}$ is larger than $\delta$. Hence, $r+\psi(\delta)>\delta$ which implies that $\e(\delta)\in (0,2)$ for every $\delta>0$. Moreover, since $\lim_{\delta\to 0+}\e(\delta)=0$ and $\e(\delta)>\tfrac{2\delta}{r+\delta}\to 2$ as $\delta\to\infty$, we infer that $\e(\delta)$ assumes each value in $(0,2)$ as $\delta\in (0,\infty)$. 

For an arbitrarily fixed $\delta>0$, define $a=r+\psi(\delta)$; we are going to show that 
\begin{equation}\label{szlenk_L}
\frac{x_0^\ast}{a}\in s_{\e(\delta)}B_{X^\ast},\,\mbox{ whereas }\,\,\,\, \frac{y_0^\ast}{a}\not\in s_{\e(\delta)}B_{X^\ast}.
\end{equation}
First, by condition (i), for every weak$^\ast$-open neighborhood $U$ of $x_0^\ast$ we have $\tfrac{x_0^\ast}{a}\pm\tfrac{v_n^\ast}{a}\in U\cap B_{X^\ast}$ for sufficiently large $n\in\N$. Note also that
$$
\mathrm{diam} (U\cap B_{X^\ast})\geq 2\Big\|\frac{v_n^\ast}{a}\Big\|>\e(\delta),
$$
which gives the first assertion in \eqref{szlenk_L}. 

Now, suppose that $\tfrac{y_0^\ast}{a}\in s_{\e(\delta)}B_{X^\ast}$. Then, there exists a~weak$^\ast$-null sequence $(w_n^\ast)_{n=1}^\infty\subset X^\ast$ with $\n{\tfrac{y_0^\ast}{a}+w_n^\ast}\leq 1$ for each $n\in\N$ and such that for every $k\in\N$ there are $m,n>k$ with $\n{w_m^\ast-w_n^\ast}>\e(\delta)$. By passing to a~subsequence, we may assume that $\n{w_n^\ast}>\tfrac{1}{2}\e(\delta)$, and hence $\n{aw_n^\ast}>\tfrac{1}{2}a\e(\delta)=\delta$ for each $n\in\N$. According to condition (ii), we have $\n{y_0^\ast+aw_n^\ast}\geq r+\phi(\delta)$ for some $m\in\N$, which yields
$$
\Big\|\frac{y_0^\ast}{a}+w_n^\ast\Big\|\geq \frac{r+\phi(\delta)}{r+\psi(\delta)}>1,
$$
a~contradiction. Therefore, we have shown the second assertion in \eqref{szlenk_L} which completes the proof.    
\end{proof}

For any $A,B>0$, define $\Ma_{A,B}(t)=At^4+Bt^2$ ($t\geq 0$). Plainly, $\Ma_{A,B}$ is a~non-degenerate Orlicz function satisfying the $\Delta_2$-condition at zero, hence the canonical unit vectors $(e_n)_{n=1}^\infty$ form a~basis of $\ell_{\Ma_{A,B}}$ (for necessary details on Orlicz sequence spaces, we refer to \cite[Ch.~4]{LT}). It is easy to see that for every $x=(x_n)_{n=1}^\infty\in \ell_{\Ma_{A,B}}$ the norm $\n{x}_{\Ma_{A,B}}$ is the nonnegative solution $\lambda $ of the equation
$$
\lambda^4-B\lambda^2\sum_{n=1}^\infty\abs{x_n}^2-A\sum_{n=1}^\infty\abs{x_n}^4=0
$$
that is, $\n{x}_{\Ma_{A,B}}=\sqrt{\tfrac{1}{2}f(x)}$, where 
$$
f(x)=B\n{x}_2^2+\sqrt{B^2\n{x}_2^4+4A\n{x}_4^4}.
$$
Since $B\n{x}_2^2\leq f(x)\leq (B+\sqrt{B^2+4A})\n{x}_2^2$, the space $\ell_{\Ma_{A,B}}$ is isomorphic to $\ell_2$, hence it is reflexive and the canonical unit vector basis $(e_n^\ast)_{n=1}^\infty$ of the (pre)dual of $\ell_{\Ma_{A,B}}$ is shrinking. In particular, a~sequence $(u_n)_{n=1}^\infty$ in $\ell_{\Ma_{A,B}}$ is weak$^\ast$-null if and only if it is bounded and $\lim_n e_j^\ast(u_n)=0$ for every $j\in\N$.

Using the classical Karush--Kuhn--Tucker (KKT) optimization method (see, e.g., \cite[\S 5.5.3]{BV}) we are going to verify that the conditions of Lemma~\ref{test_L} are satisfied for $X=\ell_{\Ma_{A,B}}^\ast$. The following obvious lemma will be used a~few times, so we record it separately for clarity.
\begin{lemma}\label{increasing_L}
For any fixed $0<a<b$, the function $[0,\infty)\ni t\mapsto h(t)=\sqrt{a+t}-\sqrt{b+t}$ is strictly increasing.    
\end{lemma}
\begin{proof}
Notice that $h^\prime(t)=\tfrac{1}{2}((a+t)^{-1/2}-(b+t)^{-1/2})>0$ for every $t>0$.  
\end{proof}

\begin{proposition}
For all $A,B>0$, and any $\e\in (0,2)$, the Szlenk derivation $s_\e B_{\ell_{\Ma_{A,B}}}$ is not a~ball. In particular, the space $\ell_{\Ma_{A,B}}^\ast$
fails the {\rm (}$M^\ast${\rm )} property.
\end{proposition}

\begin{proof}
Without loss of generality we may and do assume that $B=1$. For simplicity, we also write $\tfrac{1}{4}A$ instead of $A$, where $A>0$ is arbitrarily fixed. We shall then write $\n{\cdot}$ for the norm in $\ell_\Ma$, where $\Ma(t)=\tfrac{1}{4}At^4+t^2$ ($t\geq 0$), that is, 
$$
\n{x}=\sqrt{\tfrac{1}{2}f(x)}\,,\quad\mbox{ where }\,\, f(x)=\n{x}_2^2+\sqrt{\n{x}_2^4+A\n{x}_4^4}.
$$
Set $x_0=e_1$ and, for any $n\in\N$ let $y_0^{(n)}=\alpha_n(e_1+\ldots+e_n)$, where $\alpha_n>0$ is picked so that $\n{x_0}=\n{y_0^{(n)}}$, that is
$$
\alpha_n=\sqrt{\frac{1+\sqrt{1+A}}{n+\sqrt{n^2+An}}}.
$$
Notice that 
\begin{equation}\label{n_alpha}
    \lim_{n\to\infty}n\alpha_n^2=\frac{1+\sqrt{1+A}}{2}>1.
\end{equation}

\vspace*{2mm}\noindent
{\bf Claim. }For every $n\in\N$ and $u\in\ell_{\Ma}$ with $\mathrm{supp}(u)\subseteq \{n+1,n+2,\ldots\}$, we have 
$$
f\big(y_0^{(n)}+u\big)\geq V(\n{u}), 
$$
where the function $V\colon [0,\infty)\to [0,\infty)$ is given by
$$
V(x)=n\alpha_n^2+\frac{2}{1+\sqrt{1+A}}x^2+\sqrt{(n^2+An)\alpha_n^4+\frac{4n\alpha_n^2}{1+\sqrt{1+A}}x^2+\frac{4(1+A)}{(1+\sqrt{1+A})^2}x^4}.
$$

\vspace*{2mm}\noindent
{\it Proof of the Claim. }Consider any $u\in\ell_\Ma$ as above with fixed norm, say $2\n{u}^2=\mu$, i.e. with a~fixed value $f(u)=\mu$. Define variables $s=\n{u}_2^2$ and $t=\n{u}_4^2$. Observe that since the $\n{\!\cdot\!}_2$-norm dominates $\n{\!\cdot\!}_4$, the values of $s$ and $t$ vary over all nonnegative numbers satisfying $t\leq s$. Therefore, our aim is to solve the optimization problem of minimizing  the function 
$$
F(s,t)=n\alpha_n^2+s+\sqrt{(n^2+An)\alpha_n^4+2n\alpha_n^2s+s^2+At^2}
$$
subject to the constraints:
$$
\left\{\begin{array}{ll}
h_1(s,t)\coloneqq -t & \leq 0\\
h_2(s,t)\coloneqq t-s & \leq 0\\
g(s,t)\coloneqq s+\sqrt{s^2+At^2}-\mu & =0,
\end{array}\right.
$$
where $\mu>0$ is a fixed constant. The corresponding Lagrangian function is given by
\begin{equation*}
\begin{split}
\mathcal{K}(s,t) &=F(s,t)+\lambda g(s,t)+\mu_1 h_1(s,t)+\mu_2 h_2(s,t)\\
&=F(s,t)+\lambda g(s,t)-\mu_2s+(\mu_2-\mu_1)t.
\end{split}
\end{equation*}
If the function $F$ attains the minimal value on the set
$$
\{(s,t)\in\R^2\colon g(s,t)=0,\, h_i(s,t)\leq 0\mbox{ for }i=1,2\}
$$
at a certain point $(s,t)$, then the necessary KKT conditions say that for some $\lambda\in\R$ and $\mu_1, \mu_2\geq 0$ we have
\begin{equation}\label{system_KKT}
\left\{\begin{array}{ll}
\nabla \mathcal{K}(s,t) &=(0,0)\\
\mu_i h_i(s,t) &=0\,\,\mbox{ for }i=1,2.
\end{array}\right.
\end{equation}
In order to solve \eqref{system_KKT} we consider several cases.

\vspace*{2mm}\noindent
\underline{{\it Case 1$^\circ$}}\,\, $\mu_1=\mu_2=0$. Here, we distinguish two subcases.

\vspace*{1mm}\noindent
{\it 1$^\circ$a})\, $t=0$. Since $g(s,t)=0$, we obtain $s=\tfrac{1}{2}\mu$ which results in the value
\begin{equation}\label{V1_value}
V_1=F(\tfrac{1}{2}\mu,\tfrac{1}{2}\mu)=n\alpha_n^2+\tfrac{1}{2}\mu+\sqrt{(n^2+An)\alpha_n^4+n\alpha_n^2\mu+\tfrac{1}{4}\mu^2}.
\end{equation}

\vspace*{2mm}\noindent
{\it 1$^\circ$b})\, $t\neq 0$. We will show that this leads to a~contradiction. Indeed, $\nabla\mathcal{K}(s,t)=(0,0)$ and our assumption that $t\neq 0$ yield 

\begin{equation}\label{case_2_system}
\left\{\begin{array}{ll}
\displaystyle{1+\frac{n\alpha_n^2+s}{P(s,t)}+\lambda+\frac{\lambda s}{\sqrt{s^2+At^2}}} &=0\\[8pt]
\displaystyle{\frac{1}{P(s,t)}+\frac{\lambda}{\sqrt{s^2+At^2}}} &=0,
\end{array}\right.
\end{equation}

\vspace*{1mm}\noindent
where $P(s,t)=\sqrt{(n^2+An)\alpha_n^4+2n\alpha_n^2s+s^2+At^2}$. Combining these two equations and using the constraint $g(s,t)=0$, i.e. $\sqrt{s^2+At^2}=\mu-s$, we obtain
$$
\lambda=\frac{\mu-s}{n\alpha_n^2-\mu+s}.
$$
Plugging it into the first equation of \eqref{case_2_system} we get
\begin{equation}\label{two_stars}
1+\frac{n\alpha_n^2+s}{P(s,t)}+\frac{\mu}{n\alpha_n^2-\mu+s}=0.
\end{equation}
We also have
\begin{equation*}
\begin{split}
    P(s,t) &=\sqrt{(n^2+An)\alpha_n^4+2n\alpha_n^2s+(\mu-s)^2}\\
            &=\sqrt{(n\alpha_n^2-\mu+s)^2+An\alpha_n^4+2\mu n\alpha_n^2}.
\end{split}    
\end{equation*}
Hence, substituting $\xi=n\alpha_n^2-\mu+s$ and $C=An\alpha_n^4+2\mu n\alpha_n^2$, we have $P(s,t)=\sqrt{\xi^2+C}$. Combining it with \eqref{two_stars} we obtain $(\xi+\mu)(\sqrt{\xi^2+C}+\xi)=0$, which is impossible as $\xi+\mu=n\alpha_n^2+s>0$ and $C>0$.

\vspace*{2mm}\noindent
\underline{{\it Case 2\,$^\circ$}}\,\, $h_1(s,t)=\mu_2=0$. Since $t=0$, we again obtain the value $V_1$ given by \eqref{V1_value}.

\vspace*{2mm}\noindent
\underline{{\it Case 3\,$^\circ$}}\,\, $\mu_1=h_2(s,t)=0$. Here $s=t$, thus the condition $g(s,t)=0$ yields $s=t=\mu/(1+\sqrt{1+A})$ which result in a~new value $F(s,t)$ which equals
\begin{equation}\label{V2_value}
    V_2=n\alpha_n^2+\frac{\mu}{1+\sqrt{1+A}}+\sqrt{(n^2+An)\alpha_n^4+\frac{2n\alpha_n^2\mu}{1+\sqrt{1+A}}+\frac{(1+A)\mu^2}{(1+\sqrt{1+A})^2}}.
\end{equation}

\vspace*{2mm}\noindent
\underline{{\it Case 4$^\circ$}}\,\, $h_1(s,t)=h_2(s,t)=0$. This gives $s=t=0$ leads to a~straightforward contradiction, as $g(s,t)=0$ would then imply $\mu=0$.

\vspace*{2mm}
We have shown that the minimal value of $F(s,t)$ under our constraints equals $\min\{V_1,V_2\}$, where $V_1$ and $V_2$ are given by \eqref{V1_value} and \eqref{V2_value}. Comparing the expressions under the square root signs we observe that 
$$
n\alpha_n^2\mu>\frac{2n\alpha_n^2\mu}{1+\sqrt{1+A}}\quad\mbox{ and }\quad \frac{1}{4}\mu^2<\frac{(1+A)\mu^2}{(1+\sqrt{1+A})^2}.
$$
By Lemma~\ref{increasing_L}, we have
\begin{equation*}
\begin{split}
    V_1-V_2 &> \Big(\frac{1}{2}-\frac{1}{1+\sqrt{1+A}}\Big)\mu+\sqrt{(n^2+An)\alpha_n^4+\frac{2n\alpha_n^2\mu}{1+\sqrt{1+A}}+\frac{1}{4}\mu^2}\\
    &\hspace*{121pt}-\sqrt{(n^2+An)\alpha_n^4+\frac{2n\alpha_n^2\mu}{1+\sqrt{1+A}}+\frac{(1+A)\mu^2}{(1+\sqrt{1+A})^2}}\\
    &>\Big(\frac{1}{2}-\frac{1}{1+\sqrt{1+A}}\Big)\mu+\frac{1}{2}\mu-\frac{\sqrt{1+A}}{1+\sqrt{1+A}}\mu=0.
\end{split}    
\end{equation*}
Hence, $V_2<V_1$, which shows that the minimal value equals $V_2$. Recall that $\mu=2\n{u}^2$, so that we have 
$$
f\big(y_0^{(n)}+u\big)=F(\n{u}_2^2,\n{u}_4^2)\geq V_2=V(\n{u}),
$$
where $V(x)$ is defined as in the statement of our Claim.

\vspace*{2mm}
Now, using \eqref{n_alpha}, pick $n\in\N$ so large that $n\alpha_n^2>1$; we keep that $n$ fixed for the rest of the proof. Consider a~weak$^\ast$-null sequence of the form $(\tau e_k)_{k=2}^\infty\subset\ell_\Ma$, and note that
$$
f(x_0+\tau e_k)=1+\tau^2+\sqrt{1+A+2\tau^2+(1+A)\tau^4}.
$$
Since $\tau^2=2\n{\tau e_k}^2/(1+\sqrt{1+A})$, we can rewrite the last equation as $f(x_0+\tau e_k)=U(\n{\tau e_k})$, where the function $U\colon [0,\infty)\to [0,\infty)$ is given by
$$
U(x)=1+\frac{2}{1+\sqrt{1+A}}x^2+\sqrt{1+A+\frac{4}{1+\sqrt{1+A}}x^2+\frac{4(1+A)}{(1+\sqrt{1+A})^2}x^4}.
$$

Notice that $U(x)<V(x)$ for every $x>0$. Indeed, since $n\alpha_n^2>1$, we see by comparing the respective terms in the definitions of $V(x)$ and $U(x)$ that it suffices to consider the case where $(n^2+An)\alpha_n^4<1+A$ for some $x>0$. But then, in view of Lemma~\ref{increasing_L}, we have
\begin{equation*}
\begin{split}
    V(x)-U(x) &> n\alpha_n^2-1+\sqrt{(n^2+An)\alpha_n^4+\frac{4}{1+\sqrt{1+A}}x^2+\frac{4(1+A)}{(1+\sqrt{1+A})^2}x^4}\\
    &\hspace*{88pt}-\sqrt{1+A+\frac{4}{1+\sqrt{1+A}}x^2+\frac{4(1+A)}{(1+\sqrt{1+A})^2}x^4}\\
    &>n\alpha_n^2-1+\sqrt{n^2+An}\alpha_n^2-\sqrt{1+A}=0.
\end{split}    
\end{equation*}

\vspace*{1mm}
Now, we are going to verify the assumptions of Lemma~\ref{test_L}. First, denote
$$
r=\n{x_0}=\n{y_0^{(n)}}=\sqrt{\frac{1+\sqrt{1+A}}{2}}
$$
and observe that by the subadditivity of the square root function, we have 
\begin{equation*}
\begin{split}
U(x) &<1+\frac{2}{1+\sqrt{1+A}}x^2+\sqrt{1+A}+\frac{2}{\sqrt{1+\sqrt{1+A}}}x+\frac{2\sqrt{1+A}}{1+\sqrt{1+A}}x^2\\
&=2r^2+\frac{2}{\sqrt{1+\sqrt{1+A}}}x+2x^2<2r^2+4\sqrt{\frac{1+\sqrt{1+A}}{2}}x+2x^2=2(r+x)^2.
\end{split}
\end{equation*}
Choose any continuous, strictly increasing functions $\w U, \w V\colon [0,\infty)\to [0,\infty)$ such that 
\begin{equation}\label{in_between}
U(x)<\w U(x)<2(r+x)^2<\w V(x)<V(x)\quad\,\, (x>0).
\end{equation}
Define maps $\phi,\psi\colon (0,\infty)\to (0,\infty)$ by the formulas
$$
\phi(\delta)=\sqrt{\frac{\w V(\delta)}{2}}-r,\quad \psi(\delta)=\sqrt{\frac{\w U(\delta)}{2}}-r\quad\,\, (\delta>0).
$$
In view of \eqref{in_between}, we have $\psi(\delta)<\min\{\delta,\phi(\delta)\}$ for every $\delta>0$. Note also that for any fixed $\delta>0$ and any $\tau>0$, we have 
$$
\n{x_0+\tau e_k}=\sqrt{\frac{f(x_0+\tau e_k)}{2}}=\sqrt{\frac{U(\n{\tau e_k})}{2}}<r+\psi(\n{\tau e_k})\quad (k\geq 2).
$$
Hence, for an appropriate $\tau$ such that $\n{\tau e_k}>\delta$, we still have $\n{x_0+\tau e_k}<r+\psi(\delta)$ for $k\geq 2$, which shows that condition (i) of Lemma~\ref{test_L} is satisfied with $x_0$ in the place of $x_0^\ast$ and with the weak$^\ast$-null sequence $(\tau e_k)_{k=2}^\infty$.

Finally, note that for any weak$^\ast$-null sequence $(u_k)_{k=1}^\infty\subset \ell_\Ma$ with $\n{u_k}>\delta^\prime>\delta$ for each $k\in\N$, and arbitrarily small $\eta\in (0,\delta)$, we may find $m\in\N$ and $u\in\ell_\Ma$ with $\supp(u)\subseteq\{n+1,n+2,\ldots\}$ such that $\n{u_m-u}<\eta$. By our Claim, we have
$$
\n{y_0^{(n)}+u}=\sqrt{\frac{f\big(y_0^{(n)}+u\big)}{2}}\geq\sqrt{\frac{V(\n{u})}{2}}>r+\phi(\n{u}),
$$
hence $\n{y_0^{(n)}+u_m}>r+\phi(\n{u})-\eta>r+\phi(\delta^\prime-\eta)-\eta$. Since $\delta^\prime>\delta$ and $\eta>0$ could be chosen arbitrarily small, we may guarantee that for some $m\in\N$ we have $\n{y_0^{(n)}+u_m}>r+\phi(\delta)$. This shows that condition (ii) of Lemma~\ref{test_L} is satisfied with $y_0^{(n)}$ in the place of $y_0^\ast$.
\end{proof}

\vspace*{2mm}\noindent
{\bf Acknowledgement. }The first-named author acknowledges with gratitude the support from the National Science Centre, grant OPUS 19, project no.~2020/37/B/ST1/01052.


\begin{thebibliography}{99}

\bibitem{AT} S.A. Argyros, S. Todorcevic, \emph{Ramsey methods in analysis}, Advanced Courses in Mathematics, CRM Barcelona, Birkh\"auser 2005.
\bibitem{baernstein} A. Baernstein II, \emph{On reflexivity and summability}, Studia Math.~{\bf 42} (1972), 91--94.

\bibitem{BOS} J. Bana\'s, L. Olszowy, K.~Sadarangani, \emph{Moduli of near convexity of the Baernstein space}, Proc. Amer. Math. Soc.~{\bf 123} (1995), 3693--3699. 

\bibitem{BV} S. Boyd, L. Vandenberghe, \emph{Convex optimization}, Cambridge University Press, Cambridge 2004.


\bibitem{CS} P.G. Casazza, Th.J. Shura, \emph{Tsirelson's space}, Lecture Notes in Mathematics~1363, Springer-Verlag, Berlin--Heidelberg~1989.
\bibitem{causey} R.M. Causey, \emph{Power type $\xi$-asymptotically uniformly smooth norms}, Trans. Amer. Math. Soc.~{\bf 371} (2019), 1509--1546.

\bibitem{CDDK2} M. C\'uth, M. Dole\v{z}al, M.~Doucha, O.~Kurka, \emph{Polish spaces of Banach spaces: Complexity of isometry and isomorphism classes}, J. Inst. Math. Jussieu, online, doi:10.1017/S1474748023000440, pp. 1--39.

\bibitem{DGZ} R. Deville, G. Godefroy, V.~Zizler, \emph{Smoothness and renormings in Banach spaces}, John Wiley \& Sons, Inc., New York 1993.

\bibitem{JFA} S. Draga, T. Kochanek,
\emph{Direct sums and summability of the Szlenk index}, J.~Funct. Anal.~{\bf 271} (2016), 642--671.

\bibitem{PAMS} S. Draga, T. Kochanek,
\emph{The Szlenk power type and tensor products of Banach spaces}, Proc. Amer. Math. Soc.~{\bf 145} (2017), 1685--1698. 

\bibitem{figiel_johnson} T. Figiel, W.B. Johnson, \emph{A~uniformly convex Banach space which contains no $\ell_p$}, Compositio Math.~{\bf 29} (1974), 179--190.


\bibitem{GKL} G. Godefroy, N.J. Kalton, G. Lancien, \emph{Szlenk indices and uniform homeomorphisms}, Trans. Amer. Math. Soc.~{\bf 353} (2001), 3895--3918.



\bibitem{kalton} N.J. Kalton, \emph{$M$-ideals of compact operators}, Illinois J.~Math.~{\bf 37} (1993), 147--169.

\bibitem{KW} N.J. Kalton, D. Werner, \emph{Property $(M)$, $M$-ideals, and almost isometric structure of Banach spaces}, J.~Reine Angew. Math.~{\bf 461} (1995), 137--178. 

\bibitem{KOS} H. Knaust, E. Odell, Th. Schlumprecht, \emph{On asymptotic structure, the Szlenk index and UKK properties in Banach spaces}, Positivity~{\bf 3} (1999), 173--199.

\bibitem{KM} T. Kochanek, M. Miarka, \emph{Enveloping balls of Szlenk derivations}, Indagat. Math., published online,\\ https://doi.org/10.1016/j.indag.2024.08.007


\bibitem{lancien} G. Lancien, \emph{A~survey on the Szlenk index and some of its applications}, Rev. R.~Acad. Cien. Serie~A. Mat. {\bf 100} (2006), 209--235.

\bibitem{LT} J. Lindenstrauss, L. Tzafriri, \emph{Classical Banach spaces I. Sequence spaces}, Springer-Verlag 1977.

\bibitem{NP} I. Namioka, R.R.~Phelps, \emph{Banach spaces which are Asplund spaces}, Duke Math. J.~{\bf 42} (1975), 735--750.

\bibitem{partington} J.R. Partington,
\emph{On nearly uniformly convex Banach spaces}, Math. Proc. Cambridge Phil. Soc.~{\bf 93} (1983), 127--129. 

\bibitem{schlumprecht} T. Schlumprecht, \emph{An arbitrarily distortable Banach space}, Israel J.~Math.~{\bf 76} (1991), 81--95.

\bibitem{seifert} C.J. Seifert, \emph{The dual of Baernstein's space and the Banach--Saks property}, Bull. Acad. Polon. Sci. S\'er. Sci. Math. Astronom. Phys.~{\bf 26} (1978), 237--239. 


\bibitem{szlenk} W. Szlenk, \emph{The non-existence of a~separable reflexive Banach space universal for all separable reflexive Banach spaces}, Studia Math.~{\bf 30} (1968), 53--61.

\bibitem{tsirelson} B.S. Tsirelson, \emph{Not every Banach space contains an imbedding of $l_p$ or $c_0$}, Funct. Anal. Appl.~{\bf 8} (1974), 138--141.
\end{thebibliography}
\end{document}